\documentclass[12pt,reqno]{amsart}
\usepackage{amssymb,amsmath,amsthm,enumerate,verbatim,bbm}
\usepackage[a4paper]{geometry}

\DeclareMathOperator{\sign}{sign}
\DeclareMathOperator{\Ran}{Ran}

\DeclareMathOperator{\Ker}{Ker}

\DeclareMathOperator{\spec}{spec}

\DeclareMathOperator{\diag}{diag}

\DeclareMathOperator{\supp}{supp}

\DeclareMathOperator{\sing}{sing}

\newcommand{\abs}[1]{\lvert#1\rvert}
\newcommand{\Abs}[1]{\left\lvert#1\right\rvert}

\newcommand{\jap}[1]{\langle#1\rangle}

\newcommand{\bbT}{{\mathbb T}}
\newcommand{\bbR}{{\mathbb R}}
\newcommand{\bbC}{{\mathbb C}}

\newcommand{\bbN}{{\mathbb N}}
\newcommand{\bbZ}{{\mathbb Z}}

\newcommand{\wh}{\widehat}

\newcommand{\ess}{\text{ess}}
\newcommand{\ext}{\text{ext}}
\newcommand{\inv}{W}
\newcommand{\binv}{\mathbf{W}}

\newcommand{\bb}{\mathbf{b}}
\newcommand{\bq}{\mathbf{q}}
\newcommand{\bh}{{\mathbf{h}}}
\newcommand{\bg}{{\mathbf{g}}}

\newcommand{\bu}{{\mathbf{u}}}
\newcommand{\bsigma}{{\boldsymbol{\sigma}}}

\newcommand{\bP}{\mathbf{P}}

\newcommand{\bV}{\mathbf{V}}

\newcommand{\bR}{\mathbf{R}}

\newcommand{\bomega}{\boldsymbol{\omega}}

\newcommand{\calH}{{\mathcal H}}

\newcommand{\calF}{\mathcal{F}}

\newcommand{\calB}{\mathcal{B}}
\newcommand{\calU}{\mathcal{U}}

\newcommand{\Sch}{\mathbf{S}}

\numberwithin{equation}{section}

\theoremstyle{plain}
\newtheorem{theorem}{\bf Theorem}[section]
\newtheorem*{theorem*}{Theorem 1.1$'$}
\newtheorem{lemma}[theorem]{\bf Lemma}

\newtheorem{corollary}[theorem]{\bf Corollary}

\theoremstyle{definition}

\theoremstyle{remark}
\newtheorem*{remark*}{\bf Remark}
\newtheorem{remark}[theorem]{\bf Remark}

\newcommand{\wt}{\widetilde}
\newcommand{\eps}{\varepsilon}
\newcommand{\loc}{\mathrm{loc}}

\newcommand{\1}{\mathbbm{1}}

\newcommand{\Hank}{ H}
\newcommand{\bHank}{{ \mathbf{H}}} 
\newcommand{\Gank}{{\Gamma}}
\newcommand{\bGank}{{\mathbf{\Gamma}}}

\renewcommand{\[}{\begin{equation}}
\renewcommand{\]}{\end{equation}}

\begin{document}

\title{Localization principle for compact Hankel operators}

\author{Alexander Pushnitski}
\address{Department of Mathematics, King's College London, Strand, London, WC2R~2LS, U.K.}
\email{alexander.pushnitski@kcl.ac.uk}

\author{Dmitri Yafaev}
\address{Department of Mathematics, University of Rennes-1,
Campus Beaulieu, 35042, Rennes, France}
\email{yafaev@univ-rennes1.fr}



\begin{abstract}
In the power scale, the asymptotic behavior of the singular values of a compact Hankel operator 
is determined by the behavior of the symbol in a neighborhood of its singular support. 
In this paper, we discuss the localization principle which says that the contributions 
of disjoint parts of the singular support of the symbol to the asymptotic behavior of the singular values are independent of each other. 
We apply this principle to Hankel integral operators 
and to infinite Hankel matrices. 
In both cases, we describe a wide class of Hankel operators with  
power-like asymptotics of singular values. The leading term of this asymptotics is found explicitly.
\end{abstract}

\maketitle

\section{Introduction and main results}\label{sec.z}

\subsection{Hankel operators on the unit circle}

Hankel  operators admit various unitarily equivalent descriptions. We start
by recalling the definition of Hankel operators on the Hardy class $H^2(\bbT)$. 
Here  $\bbT$ is the unit circle in the complex plane, equipped with the 
normalized Lebesgue measure $dm(\mu)= (2\pi i \mu)^{-1} d\mu$, $\mu\in\bbT$;
the Hardy class $H^2(\bbT)\subset L^2(\bbT)$ is defined in the standard way as
the subspace of $L^2(\bbT)$ spanned by the functions $1, \mu, \mu^{2 },\dots$. 
Let $P_+: L^2(\bbT)\to H^2(\bbT)$ be the orthogonal projection onto $H^2(\bbT)$, 
and let $\inv$   be  the involution in $L^2(\bbT)$ defined by $(\inv  f)(\mu)=f(\bar{\mu})$. 
For a function  $\omega \in L^\infty (\bbT)$, which is called a \emph{symbol} in this context, 
the \emph{Hankel operator} $\Hank (\omega)$  is defined by the relation
\begin{equation}
\Hank(\omega)f=P_+(\omega \inv  f).
\label{eq:HH}
\end{equation}
Background information on the theory of Hankel  operators can be found e.g. in  
the books \cite{NK,Peller}.

Recall that the singular values of a compact 
operator $\Hank$ are defined by the relation 
$s_{n} (\Hank)=\lambda_{n} (\abs{\Hank})$, 
where $\{\lambda_{n} (\abs{\Hank})\}_{n=1}^\infty$
is the non-increasing sequence of  eigenvalues of the compact 
positive operator $\abs{\Hank}=\sqrt{\Hank^*\Hank}$ 
(enumerated with multiplicities taken into account).
The study of singular values of compact Hankel operators has a long
history and is linked to rational approximation, 
control theory and other subjects, see,  e.g. \cite{Peller}.  
In fact, this paper is in part motivated by its applications in \cite{Rat} 
to the rational approximation of functions with logarithmic singularities.

Singular values $s_{n} (H(\omega))$ of a Hankel operator with a symbol $\omega \in C^\infty({\bbT})$ 
decay faster than any power of $n^{-1}$ as $n\to\infty$. 
On the other hand, singularities of   $\omega$ generate a slower decay of singular values.
Here we will be interested in the case when the singular values behave  as 
some power of $n^{-1}$. 
Optimal upper estimates on singular values of Hankel operators are due to V.~Peller, 
see  \cite[Section 6.4]{Peller}. 
He found necessary and sufficient conditions on $\omega$ for the estimate
$$
s_{n} (H(\omega))\leq C n^{-\alpha}
$$
for some $\alpha>0$. These conditions are stated in terms of the Besov-Lorentz classes.

It is natural to expect that the asymptotic behavior of singular values is determined  by the 
behavior of the symbol $\omega$ in a neighborhood of its singular support. 
We justify this thesis and show that the contributions of the disjoint components of 
the singular support of $\omega$ to the asymptotics of the singular values of 
$\Hank(\omega)$ are independent of each other.  
We use the term ``localization principle" for this  fact. This principle is well 
understood in the context of the study of the essential spectrum \cite{Power}  and of the absolutely continuous spectrum
\cite{Howland1}   of non-compact Hankel operators.
Our aim here is to bring this principle to the fore in the question of the 
asymptotics of singular values of compact Hankel operators.

In our applications  the singular support 
of  $\omega $ consists of a finite number of points. We use the results of our previous publication \cite{II} (where the history of the problem is described) on the  asymptotic behavior of eigenvalues of certain classes of self-adjoint Hankel operators. The localization principle  allows us to combine the contributions of different singular points and thus 
to determine  the asymptotics of singular values for a wider (compared to  \cite{II}) class of Hankel operators.
In particular, for Hankel  matrices with oscillating matrix elements  
  we show that the contributions of different 
oscillating  terms to  the asymptotics of singular values are independent of each other.
We also establish similar results for Hankel integral operators whose integral kernels have a  singularity at some finite point $t_{0}\geq 0$ and several  oscillating terms at infinity.

\subsection{Localization principle}

Recall that the singular support $\sing\supp\omega$ of a function 
$\omega\in L^\infty (\bbT)$ is defined as the smallest closed set $X\subset \bbT$ 
such that $\omega\in C^\infty (\bbT\setminus X)$.
Localization principle for Hankel operators \eqref{eq:HH} is stated as follows.

\begin{theorem}\label{LOCAL}
Let $\omega_1,\omega_2,\dots,\omega_L$ be bounded functions on $\bbT$ such that
 \[
\sing\supp\omega_{\ell} \cap \sing\supp\omega_j =\varnothing, \quad \ell\neq j.
\label{eq:singsupp}
\]
Set $\omega=\omega_1+\dots+\omega_L$. 
  Then for all $p>0$ we have the relations
\begin{align}
\limsup_{n\to\infty}n s_n(H(\omega ) )^{p}
&\leq
\sum_{\ell=1}^L
\limsup_{n\to\infty} n s_n (H(\omega_{\ell}))^{p},
\label{eq:c2aL}
\\
\liminf_{n\to\infty}n s_n(H(\omega ))^{p}
&\geq
\sum_{\ell=1}^L
\liminf_{n\to\infty} n s_n(H(\omega_{\ell}))^{p} .
\label{eq:c2bL}
\end{align}
In particular,  
\begin{equation}
\lim_{n\to\infty}n s_n(H(\omega ) )^{p}
=
\sum_{\ell=1}^L
\lim_{n\to\infty} n s_n (H(\omega_{\ell}))^{p}
\label{c2aa}
\end{equation}
provided that all limits in the right-hand side exist.
\end{theorem}

In applications, the upper and  lower limits in this  theorem usually coincide.
However, we prefer to work with these limits separately   because it is 
more general and, at the same time,  it is technically more convenient.

\subsection{Discussion}
Theorem~\ref{LOCAL} can be equivalently stated in terms of the
counting functions. For a compact operator $\Hank$, the singular value 
counting function is defined by 
\begin{equation}
n(\eps;\Hank)
=
\#\{ n: s_n (\Hank)>\eps\}, 
\quad \eps>0.
\label{eq:CF}
\end{equation}
We have 
$$
\limsup_{n\to\infty}n s_n(\Hank)^{p}
=
\limsup_{\eps\to 0}\eps^{p} n(\eps;\Hank)
$$
and similarly for the lower limits. 
Thus, focussing for simplicity on the case when the limits
in the right hand side exist and are finite, we can rewrite \eqref{c2aa} as
\begin{equation}
n(\eps;\Hank(\omega))
=
\sum_{\ell=1}^L
n(\eps;\Hank(\omega_\ell))+o(\eps^{-p}), \quad \eps\to 0.
\label{c2aaa}
\end{equation}

Our proof of Theorem~\ref{LOCAL} consists of two steps. The first one is to check that  under the 
assumption \eqref{eq:singsupp} the operators
$\Hank(\omega_\ell)$ are \emph{asymptotically  orthogonal} in the sense that
for all $j\not=\ell$ and all $\alpha>0$ we have
\[
s_n( H(\omega_{\ell}) ^* H(\omega_{j}))=O(n^{-\alpha}), 
\quad 
s_n( H(\omega_{\ell})  H(\omega_{j})^*)=O(n^{-\alpha}), \quad n\to\infty.
\label{z6}
\]
This result follows from the reduction of the products of Hankel 
operators in \eqref{z6} to integral operators in $L^2(\bbT)$ with smooth kernels.

The second step is to show that \eqref{z6} implies relations  \eqref{eq:c2aL} and  \eqref{eq:c2bL}.
This fact is not specific for Hankel operators.
In order to get some intuition into its proof, 
let us suppose for a moment that the operators $ \Hank(\omega_\ell)$
are pairwise \emph{orthogonal} in the sense that 
\begin{equation}
\Hank(\omega_j)^*\Hank(\omega_\ell)=0 
 \quad\text{ and }\quad
\Hank(\omega_j)\Hank(\omega_\ell)^*=0, 
\quad \forall j\not=\ell.
\label{z11}
\end{equation}
Then 
$$
\Ran \Hank(\omega_j)\perp\Ran \Hank(\omega_\ell)
\quad\text{ and }\quad
\Ran \Hank(\omega_j)^*\perp\Ran \Hank(\omega_\ell)^*,
\quad \forall j\not=\ell.
$$
Thus, representing the sum $\Hank(\omega)=\Hank(\omega_1)+\dots+\Hank(\omega_L)$ 
as a ``block-diagonal'' operator acting from 
$\oplus_{\ell=1}^L \overline{\Ran \Hank(\omega_\ell)^*}$  to 
$\oplus_{\ell=1}^L \overline{\Ran \Hank(\omega_\ell)}$, 
we conclude that 
$$
n(\eps;\Hank(\omega))
=
\sum_{\ell=1}^L n(\eps;\Hank(\omega_\ell)), \quad \forall\eps>0.
$$
 Of course, the orthogonality condition \eqref{z11} is too strong.  In fact,
  an operator theoretic
result,  Theorem~\ref{lma.c1}, shows that the asymptotic  orthogonality  \eqref{z6} with 
$\alpha>2/p$ ensures
the   relation \eqref{c2aaa}.

Representing Hankel operators in the basis $\{\mu^j\}_{j=0}^\infty$ in $H^2(\bbT)$, 
one obtains the class of infinite Hankel matrices of the form $\{h(j+k)\}_{j,k=0}^\infty$
in the space $\ell^2(\bbZ_+)$. 
We give an application of the localization principle to such Hankel matrices in Theorem~\ref{thm.a4}.
Although the localization principle in the form stated above (Theorem~\ref{LOCAL}) is
quite natural, this application looks far less obvious.

Theorem~\ref{LOCAL} can be equivalently stated (see Theorem~\ref{LOCALC}) 
in terms of Hankel operators $\bHank(\bomega)$ acting in the Hardy space 
$H_+^{2}({\bbR})$ of functions analytic in the upper half-plane. 
In this case the symbol $\bomega (x) $ is a   function
of $x\in {\bbR}$. This leads to new results for Hankel operators 
defined as integral operators in the space $ L^2(\bbR_+)$.

We will refer to the Hankel operators in $H^2(\bbT)$ and in $\ell^2(\bbZ_+)$
as to the discrete case, and to the Hankel operators in $H^2_+(\bbR)$ and in $L^2(\bbR_+)$
as to the continuous case. 
We will use boldface font for objects associated with the continuous case. 
We have tried to make exposition in the discrete and continuous cases parallel as much as possible. 

\subsection{Related work}\label{sec.z3}

Recall that for a bounded operator $H$, the non-zero parts of the operators 
$$
\begin{pmatrix}
\abs{H} & 0
\\
0  & -\abs{H}
\end{pmatrix}
\quad\text{ and }\quad 
\begin{pmatrix}
0 & H
\\
H^*  & 0
\end{pmatrix}
$$
are unitarily equivalent. Therefore various spectral results for $\abs{\Hank(\omega)}$ are equivalent to those for the self-adjoint Hankel operator with  the matrix valued symbol
$$
\Omega(\mu)= 
\begin{pmatrix}
0 & \omega (\mu)
\\
\overline{\omega(\bar{\mu})}  & 0
\end{pmatrix}.
$$
In particular, the study of the singular values of $H(\omega)$ is equivalent to the study of 
the eigenvalues of the Hankel operator with the symbol $\Omega(\mu)$.

Some forms of localization principle are known in the study of the continuous spectrum of $\abs{\Hank(\omega)}$.
As far as we are aware, 
the idea of separation of singularities of the symbol goes back to the work \cite{Power}
of S.~R.~Power  on the essential spectrum $\spec_{\mathrm{ess}} $ of Hankel operators   with piecewise continuous
symbols $\omega$. Let $a_{j}\in{\bbT}$ be the points where $\omega$ has the jumps
$$
\varkappa(a_{j})=
\lim_{\varepsilon\to +0} \omega (a_{j}e^{i\varepsilon})-\lim_{\varepsilon\to +0} \omega (a_{j}e^{-i\varepsilon})\neq 0.
$$
Although Power was interested in the essential spectrum of $\Hank(\omega)$ (which we do not discuss here), 
it follows from the matrix version of his results that
\begin{equation}
\spec_{\mathrm{ess}} (\abs{\Hank(\omega)})
=   
[0,M], \quad M=\tfrac12\sup_{ a_{j}\in {\bbT}}\abs{\varkappa (a_{j})},
\label{eq:XX1y}
\end{equation}
where the supremum is taken over all points $a_j$ where $\omega$ has a jump.

A description of  the absolutely continuous spectrum of $\abs{\Hank(\omega)}$ 
with piecewise continuous symbol $\omega$
follows from the matrix version of the results of Howland \cite{Howland1}, where
the trace class method of scattering theory was used. 
This question  was also 
studied in our previous paper \cite{PY1} by using  
the so-called smooth method of scattering theory. 
In both cases, 
under some mild additional   assumptions, including the condition   
that $\omega$ has finitely many jumps, it can be shown that 
\begin{equation}
\spec_{\mathrm{ac}} (\abs{\Hank(\omega)})=   \bigcup_{ a_{j}\in {\bbT}}\,  [0, \tfrac12 |\varkappa (a_{j}) | ] .
\label{eq:XX1x}
\end{equation}
Every term in the right-hand side of \eqref{eq:XX1x} gives its own band of the absolutely
continuous spectrum of  multiplicity one.
Thus, formula \eqref{eq:XX1x} can be regarded as the continuous spectrum analogue of the localisation
principle discussed in this paper: the contributions 
of different jumps of $\omega$ to $\spec_{\rm ac} (\abs{\Hank(\omega)})$ are independent of each other.  
Of course, formulas  \eqref{eq:XX1y} and  \eqref{eq:XX1x} are consistent with each other.

\subsection{The structure of the paper}
In Section~\ref{sec.y} we prove the localization principle in the discrete case (Theorem~\ref{LOCAL}) and 
also state and prove its analogue in the continuous case (Theorem~\ref{LOCALC}).
In Section~\ref{sec.b}, we describe the applications of localization principle to 
the Hankel operators acting in $\ell^2(\bbZ_+)$. The main result is stated as Theorem~\ref{thm.a4}  and its proof is given in 
Section~\ref{sec.a}. 
In Section~\ref{sec.c} we give applications to integral Hankel operators in 
$L^2(\bbR_+)$. The main result is stated as Theorem~\ref{thm.d3}  and its proof is given in Section~\ref{sec.d}. 
In Section~\ref{sec.f} we consider integral Hankel operators whose kernels have local singularities in $\bbR_{+}$.

\subsection{Some notation}
For $\omega\in L^2(\bbT)$, the Fourier coefficients of $\omega$ are denoted as usual by
$$
\wh \omega(j)= \int_{\bbT}  \omega(\mu) \mu^{-j} dm(\mu),
\quad 
j\in\bbZ.
$$
We will consistently make use of the following constant, which appears
in our asymptotic formulas: 
\begin{equation}
v(\alpha)=2^{-\alpha}
\pi^{1-2\alpha}
\big(B(\tfrac1{2\alpha},\tfrac12)\big)^\alpha, \quad \alpha>0;
\label{a7}
\end{equation}
here $B(\cdot,\cdot)$ is the Beta function. 
We make a standing assumption that the exponents 
$p>0$  and $\alpha>0$ are related by $\alpha=1/p$.

\section{Proof of localization principle}\label{sec.y}

In this section, we prove  Theorem~\ref{LOCAL} as well as a similar statement, Theorem~\ref{LOCALC}, 
for Hankel operators  in the Hardy space $H^2_+(\bbR)$ of functions analytic in the upper half-plane.

\subsection{Preliminaries}

Let $\calB$ be the algebra of bounded operators in a Hilbert space $\mathcal H$, 
and let $\Sch_{\infty}$ be the ideal of compact operators in $\calB$.
For $p>0$, the weak Schatten class $\Sch_{p,\infty}$ consists of all compact operators $A$ 
such that 
$$
\sup_{n} n s_n(A)^p<\infty.
$$
The subclass $\Sch_{p,\infty}^0\subset  \Sch_{p,\infty}$ 
is defined by the condition
$$
\lim_{n\to\infty} n s_n(A)^p =0.
$$
It is well known that both $\Sch_{p,\infty}$ and $\Sch_{p,\infty}^0$ are ideals of $\calB$; 
in particular, they are linear spaces. 
Of course $A\in\Sch_{p,\infty}$ (or $A\in\Sch_{p,\infty}^0$) if and only if the same is true for its adjoint $A^*$.
We set $\Sch_0=\cap_{p>0}\Sch_{p,\infty}$, that is, 
\begin{equation}
A\in\Sch_0\quad \Leftrightarrow\quad s_n(A)=O(n^{-\alpha}),\quad n\to\infty,\quad \forall \alpha>0.
\label{eq:sss}\end{equation}

First we recall a classical result in perturbation theory (see e.g. \cite[Theorem 11.6.8]{BSbook})
 on the spectral stability of singular values. 
\begin{lemma}\label{lma.a2}
Let $A\in\Sch_{\infty}$ and $B\in\Sch_{p,\infty}^0$ for some $p>0$. 
Then 
\begin{align}
\limsup_{n\to\infty}n s_n(A+B)^p
&=
\limsup_{n\to\infty}n s_n(A)^p,
\label{eq:Lsup}\\
\liminf_{n\to\infty}n s_n(A+B)^p
&=
\liminf_{n\to\infty}n s_n(A)^p.
\label{eq:Linf}\end{align}
\end{lemma}

Lemma~\ref{lma.a2} is stated   in a slightly more general form  than usual 
(see, e.g., Theorem 11.6.8 in \cite{BSbook}) because we do not require that 
$A\in\Sch_{p,\infty}$ and hence the limits in \eqref{eq:Lsup} and \eqref{eq:Linf} 
may be infinite; in this case Lemma~\ref{lma.a2} means that both sides in \eqref{eq:Lsup} and \eqref{eq:Linf} are infinite simultaneously. 
Note   that if $A\not\in\Sch_{p,\infty}$, then the expression \eqref{eq:Lsup} is infinite, but the 
expression \eqref{eq:Linf} may be finite. Lemma~\ref{lma.a2}
can also be equivalently stated in terms of the singular value counting functions
$n(\eps,A)$ defined by \eqref{eq:CF}.

\subsection{Asymptotically orthogonal operators}
 Note the implication 
\begin{equation}
A\in\Sch_{p,\infty}, 
\quad 
B\in\Sch_{p,\infty}
\quad 
\Rightarrow
\quad
A^*B\in\Sch_{p/2,\infty}, \quad
A B^*\in\Sch_{p/2,\infty}
\label{c14x}
\end{equation}
(see, e.g. \cite[Theorem 11.6.9]{BSbook}).
We   say that the operators $A$ and $B$ in $\Sch_{p,\infty}$ are asymptotically  orthogonal
if the class $\Sch_{p/2,\infty}$ in the right side of \eqref{c14x} can be replaced by its subclass $\Sch_{p/2,\infty}^0$.
The following theorem allows us to study singular values of sums of asymptotically  orthogonal operators.
This result is the key operator theoretic ingredient of our construction.

\begin{theorem}\label{lma.c1}
Let $p>0$.  
Assume that $A_1,\dots,A_L\in\Sch_{\infty}$ and
\begin{equation}
A_\ell^*A_j\in\Sch_{p/2,\infty}^0,
\quad 
A_\ell A_j^*\in\Sch_{p/2,\infty}^0
\quad 
\text{ for all $\ell\not=j$.}
\label{c1}
\end{equation}
Then for   $A= A_1+\dots+A_L$, we have
\begin{align}
\limsup_{n\to\infty}n s_n(A )^{p}
&\leq
\sum_{\ell=1}^L
\limsup_{n\to\infty} n s_n(A_\ell)^{p},
\label{c2a}
\\
\liminf_{n\to\infty}n s_n(A )^{p} 
&\geq
\sum_{\ell=1}^L
\liminf_{n\to\infty} n s_n(A_\ell)^{p}.
\label{c2b}
\end{align}
In particular, 
$$
\lim _{n\to\infty}n s_n(A)^{p}
=
\sum_{\ell=1}^L
\lim_{n\to\infty} n s_n(A_\ell)^{p}
$$
provided that all limits in the right-hand side exist.
\end{theorem}

\begin{proof}
Let us prove the first relation \eqref{c2a}; the second one is proven in the same way.
We argue in terms of   counting functions \eqref{eq:CF}.
For an operator $A\in\Sch_{\infty}$, let us denote
$$
\Delta_p(A)=\limsup_{\varepsilon\to  0}\varepsilon^{1/p} n(\eps;A)
$$
(this limit may be infinite).
Then our aim is to prove that 
\begin{equation}
\Delta_{p}(A)
\leq
\sum_{\ell=1}^L \Delta_{p}(A_\ell),
\label{c2c}
\end{equation}
which is \eqref{c2a} in different notation. 
Put
$$
\calH^L=\underbrace{\calH\oplus\dots\oplus\calH}_{\text{$L$ terms}}
$$
and let
$A_0=\diag\{A_1,\dots,A_L\}$ in $\calH^L$, i.e., 
$$
A_0(f_1,\dots,f_L)=(A_1f_1,\dots,A_Lf_L).
$$
Since
\begin{equation}
A_0^*A_0=\diag\{A_1^*A_1,\dots,A_L^*A_L\},
\label{eq:C6}
\end{equation}
we see that
$$
n(\eps;A_0)
=
\sum_{\ell =1}^L
n(\eps;A_\ell)
$$
and therefore, multiplying by $\eps^{1/p}$, taking $\lim\sup$ as $\eps\to0$ and using the 
subadditivity of $\lim\sup$, we obtain
\begin{equation}
\Delta_{p/2}(A_0^*A_0)
\leq
\sum_{\ell=1}^L \Delta_{p/2}(A_\ell^* A_\ell)
=
\sum_{\ell=1}^L \Delta_{p}(A_\ell).
\label{c6}
\end{equation}

Next, let 
$J:\calH^L\to\calH$ be the operator given by 
$$
J(f_1,\dots,f_L)=f_1+\dots+f_L
\quad \mbox{so that} \quad 
J^* f=(f,\dots,f).
$$
Then
$$
J A_0(f_1,\dots,f_L)=A_1f_1 + \dots +A_Lf_L
$$
and
$$
(J A_0)^* f=(A_1^*f,\dots,A_L^*f).
$$
It follows that
\begin{equation}
(J A_{0}) (J A_0)^* f= (A_1 A_1^* +\dots +A_L A_L^*) f
\label{eq:cc}
\end{equation}
and the operator $(J A_0)^*(J A_0)$ is a ``matrix'' in $\calH^L$ given by
\begin{equation}
(J A_0)^*(J A_0)
=
\begin{pmatrix}
A_1^*A_1 & A_1^*A_2 & \dots & A_1^* A_L
\\
A_2^*A_1 & A_2^*A_2 & \dots & A_2^* A_L
\\
\vdots & \vdots & \ddots & \vdots
\\
A_L^*A_1 & A_L^*A_2 & \dots  & A_L^* A_L
\end{pmatrix}.
\label{eq:CC}
\end{equation}

According to \eqref{eq:C6} and \eqref{eq:CC}  we have
\[
(J A_0)^*(J A_0)-A_0^*A_0\in\Sch_{p/2,\infty}^0.
\label{c3a}
\]
Indeed,   the ``matrix'' of the operator in \eqref{c3a}  has zeros on the 
 diagonal, and its
 off-diagonal elements are given by $A_\ell^* A_j$, $\ell\not=j$. Thus    \eqref{c3a} follows from the first assumption  \eqref{c1}. 
Therefore Lemma~\ref{lma.a2} implies that
$$
\Delta_{p/2}((J A_0)^*(J A_0))
=
\Delta_{p/2}(A_0^*A_0)
$$
or
\[
\Delta_{p/2}((J A_0)(J A_0)^*)
=
\Delta_{p/2}(A_0^*A_0)
\label{c3b}
\]
because for any compact operator $T$ the non-zero singular values of $T^*T$ and $TT^*$ coincide.

Further, since
$
AA^*=
\sum_{\ell,j=1}^L A_\ell A_j^*,
$
it follows from \eqref{eq:cc} and the second assumption  \eqref{c1} that
$$
AA^*-(J A_0)(J A_0)^* = \sum_{j\not=\ell}A_\ell A_j^* \in\Sch_{p/2,\infty}^0.
$$
Using again Lemma~\ref{lma.a2}, from here we obtain
$$
\Delta_{p}(A)=\Delta_{p/2}(AA^*)=\Delta_{p/2}((J A_0)(J A_0)^*).
$$
Combining the last equality with    \eqref{c3b}, we see that
$\Delta_p(A)=\Delta_{p/2}(A_0^*A_0)$. Thus 
\eqref{c6} yields the relation \eqref{c2c}. 
\end{proof}

Under  slightly  more restrictive assumptions Theorem~\ref{lma.c1} appeared first in
\cite[Theorem 3]{BS2}.  Our  proof is quite different from that of \cite{BS2}. 

\begin{remark}
Let us mention two known statements that are similar in spirit to Theorem~\ref{lma.c1}.
Below $A_1,\dots,A_L$ are bounded operators and $A=A_1+\dots+A_L$. 

(i) If the products $A_\ell^*A_j$, $A_\ell A_j^*$ are compact for all $j \not=\ell$, then 
for the essential spectra of $A$ one has the formula
$$
\spec_\ess(A)\cup\{0\}=\bigcup_{\ell=1}^L \spec_\ess(A_\ell),
$$
see,  e.g. \cite[Section 10.1]{Peller}.

(ii) 
If $A_1,\dots,A_L$ are   self-adjoint operators such that $A_\ell A_j$ are 
trace class for all $j\not=\ell$, then the absolutely continuous part of $A$ is unitarily 
equivalent to the orthogonal sum of the absolutely continuous parts of the operators $A_\ell$.
This is known as Ismagilov's theorem, see \cite{Ismagilov}.
\end{remark}

\subsection{Proof of localization principle for Hankel operators in $H^2(\bbT)$}\label{sec.z3H}

First we state two well-known facts that will be needed for the proof of Theorem~\ref{LOCAL}    given at the end of the subsection.

We  recall that the Hankel operators $H(\omega)$ are defined by \eqref{eq:HH}; the   class $\Sch_0 $ is defined by \eqref{eq:sss}. 

\begin{lemma}\label{mm}
\begin{enumerate}[\rm (i)]
\item
Let $K$ be an integral operator in $L^2(\bbT)$ with an integral kernel of the class 
$C^\infty(\bbT\times\bbT)$. Then $K\in\Sch_0$. 
\item
Let $\omega\in C^\infty(\bbT)$; then $\Hank(\omega)\in\Sch_0$. 
\end{enumerate}
\end{lemma}

\begin{proof}
Part (i) is a classical fact; it can be obtained, for example, by approximating the integral 
kernel of $K$ by trigonometric polynomials. This yields a fast approximation of $K$ by 
finite rank operators. 

Part (ii) is also well-known; let us show that it follows from part (i). 
It will be convenient to consider the projection $P_+$ here as an operator
acting from $L^2(\bbT)$ to $L^2(\bbT)$ (rather than from $L^2 (\bbT)$ to $H^2 (\bbT)$). 
Recall that $P_+$ acts according to the formula
\begin{equation}
 (P_{+} f)(\mu)= \lim_{\epsilon\to+0}\int_{\bbT} 
\frac{f(\mu')}{\mu'- (1-\epsilon)\mu}\mu' dm(\mu'),
\label{eq:di}
\end{equation}
and that $\inv$ is the involution $(\inv  f)(\mu)=f(\bar{\mu})$. 
We have to prove that the operator $P_+\omega \inv P_+$ in $L^2(\bbT)$ belongs to 
the class $\Sch_0$. 
Since $P_+ \inv P_+$ is a rank one operator (projection onto constants), it suffices to check that 
\begin{equation}
P_+\omega \inv P_+-\omega P_+ \inv P_+ = [P_+,\omega] \inv P_+\in\Sch_0.
\label{eq:did}
\end{equation} 
It follows from \eqref{eq:di} that the commutator $[P_+,\omega]$  is an 
integral operator in $L^2(\bbT)$ with the kernel
$$
\frac{\omega(\mu')-\omega(\mu)}{\mu'-\mu} \mu', \quad \mu,\mu'\in\bbT.
$$
This is a $C^\infty$ function, and so $[P_+,\omega]\in \Sch_0$ which implies   \eqref{eq:did}.
\end{proof}

The following assertion allows us to separate the contributions of different singularities of the symbol. 
Essentially, this is a very well known argument, see, e.g. \cite{Power}.

\begin{lemma}\label{disjoint}
Let $\omega_1,\omega_2\in L^\infty(\bbT)$ be such that 
$\sing\supp \omega_1\cap\sing\supp\omega_2=\varnothing$. 
Then 
$$
\Hank(\omega_1)^*\Hank(\omega_2)\in\Sch_0, 
\quad
\Hank(\omega_1)\Hank(\omega_2)^*\in\Sch_0.
$$
\end{lemma}

\begin{proof}
Let $\varkappa_{1}$, $\varkappa_{2}$ be real functions in $C^\infty(\bbT)$ 
with disjoint supports such that 
$$
 (1-\varkappa_k) \omega_k  \in C^\infty(\bbT), 
\quad k=1,2.
$$
By Lemma~\ref{mm}(ii), we have
$$
\Hank((1-\varkappa_k) \omega_k)\in \Sch_0,
$$
and hence it suffices to show that
\begin{equation}
\Hank(\varkappa_{1} \omega_1)^*\Hank(\varkappa_2 \omega_2) \in \Sch_0,
\quad
\Hank(\varkappa_{1} \omega_1)\Hank(\varkappa_2 \omega_2)^* \in \Sch_0.
\label{eq:di1S}\end{equation}
It follows from definition \eqref{eq:HH} that
$$
\Hank(\varkappa_{1} \omega_1)^*\Hank(\varkappa_2 \omega_2) f
=
P_+ \inv  {\omega_1} ( {\varkappa_1} P_+ \varkappa_2) \omega_2 \inv f , \quad f\in H^2({\bbT}).
$$
Since the supports of $\varkappa_1$ and $\varkappa_2$ are disjoint, 
the operator $ \varkappa_1 P_+\varkappa_2$
has a $C^\infty$ smooth integral kernel
$$
\frac{ {\varkappa_1(\mu)}\varkappa_2(\mu')}{\mu'-\mu  } \mu' ,
\quad \mu,\mu'\in \bbT,
$$
and so by Lemma~\ref{mm}(i) it belongs to the class $\Sch_0$. This ensures the first 
inclusion in  \eqref{eq:di1S}. 
In view of the obvious identity 
$$
H(\omega)^* = H(\omega_{*})
\quad   \text{where}\quad  \omega_{*} (\mu) =\overline{\omega(\bar{\mu})},
$$
the second inclusion  \eqref{eq:di1S} follows from the first one.
\end{proof}

\begin{proof}[Proof of Theorem~\ref{LOCAL}]
Let us apply the abstract Theorem~\ref{lma.c1} to the Hankel operators
$A_{\ell}= H(\omega_{\ell})$, $ \ell=1,\ldots, L$. 
Lemma~\ref{disjoint} implies that the asymptotic orthogonality condition \eqref{c1} is satisfied. 
Therefore the asymptotic relations  \eqref{eq:c2aL}  and \eqref{eq:c2bL} 
follow directly from \eqref{c2a}  and \eqref{c2b}.
\end{proof}

\subsection{Hankel operators in $H^2_+(\bbR)$}
Hankel operators can also be defined  in the Hardy space $H^2_+(\bbR)$ of functions analytic in the upper half-plane.
We denote by $\Phi$ the unitary Fourier transform on $L^2(\bbR)$, 
$$
\wh u(t)=(\Phi u)(t)=\frac1{\sqrt{2\pi}} \int_{-\infty}^\infty u(x) e^{-ixt}dx.
$$
Let $H^2_+(\bbR)\subset L^2(\bbR)$ be the  Hardy class, 
$$
H^2_+(\bbR)=\{u\in L^2 (\bbR): \wh u(t)=0 \text{ for $t<0$}\},
$$
and let $\bP_+:L^2(\bbR)\to H^2_+(\bbR)$ be the corresponding orthogonal projection. 
Let   $\binv$   be the involution in $L^2(\bbR)$, $({\binv}f)(x)=f(-x)$. 
For $\bomega\in L^\infty(\bbR)$,  the  operator 
$\bHank({\bomega})$ in $H^2_+(\bbR)$
is defined by the formula
\begin{equation}
\bHank({\bomega})f=\bP_+({\bomega}{\binv}  f), 
\quad
f\in H^2_+(\bbR).
\label{e3}
\end{equation}

There is a unitary equivalence between the  Hankel operators 
$\Hank(\omega)$ defined in $H^2(\bbT)$ by formula \eqref{eq:HH} 
and the Hankel operators $\bHank({\bomega})$ defined in $H^2_+(\bbR)$ by formula \eqref{e3}. 
Indeed, let
\[
w=\frac{z-i/2}{z+i/2}, \quad z=\frac{i}{2}\frac{1+w}{1-w}
\label{b8a}
\]
be the  standard conformal map sending the upper half-plane
 onto the unit disc, and let  $\calU: H^2(\bbT)\to H^2_+ (\bbR)$ be
 the corresponding unitary operator defined by
$$
(\calU f)(x)
=
\tfrac1{\sqrt{2\pi}}\tfrac1{x+i/2} f(\tfrac{x-i/2}{x+i/2}), 
\quad
(\calU^* \mathbf f)(\mu)
=
i\sqrt{2\pi}\tfrac1{1-\mu}\mathbf  f (\tfrac{i}{2}\tfrac{1+\mu}{1-\mu}).
$$
Then
\begin{equation}
\calU \Hank(\omega) \calU^*=\bHank({\bomega}), 
\quad\text{ if }\quad
{\bomega} (x)= -\tfrac{x-i/2}{x+i/2}\omega (\tfrac{x-i/2}{x+i/2}). 
\label{eq:HG}
\end{equation}
So the localization principle stated for $H(\omega)$ can be automatically 
mapped to operators $\bHank({\bomega})$. This is discussed below. 

\subsection{Localization principle in $H^2_+(\bbR)$}

Symbols $\bomega (x)$ of Hankel operators $\bHank(\bomega)$
have the exceptional points $x=+\infty$ and $x=-\infty$;  it will be convenient 
to identify these two points.  
The real line with such identification will be denoted $\bbR_{*}$. 
We write $\bomega \in C (\bbR_{*})$ if $\bomega \in C (\bbR )$ and 
$$
\lim_{x\to \infty}\bomega (x)=\lim_{x\to - \infty}\bomega(x),
$$
where both limits are supposed to exist. 
Similarly, we write $\bomega\in C^\infty (\bbR_{*})$ if 
$\bomega\in C^\infty (\bbR )$ and, for all $m=0,1,\ldots$, 
\begin{equation}
\lim_{x\to \infty}\bomega^{(m)} (x)=\lim_{x\to - \infty}\bomega^{(m)}(x).
\label{eq:id1}
\end{equation}
In particular, the point $x=\infty$ belongs to the singular support of $\bomega$ if 
for at least one $m\geq0$ the relation \eqref{eq:id1} fails
(i.e. if either  at least  one of the limits does not exist or   the limits are not equal).

Let us state the localization principle for Hankel operators in $H^2_+(\bbR)$. 

\begin{theorem}\label{LOCALC}
Let $\bomega_{\ell}\in L^\infty (\bbR)$, $\ell=1,\ldots, L<\infty$, be such that
$$
\sing\supp\bomega_{\ell} \cap \sing\supp\bomega_j =\varnothing, \quad \ell\neq j.
$$
Set $\bomega=\bomega_{1}+\dots+\bomega_L$.
Then for all $p >0$ we have the relations
\begin{align*}
\limsup_{n\to\infty}n s_n(\bHank( {\bomega}))^{p}
&\leq
\sum_{\ell=1}^L
\limsup_{n\to\infty} n s_n (\bHank( {\bomega}_\ell))^{p},
\\
\liminf_{n\to\infty}n s_n(\bHank( {\bomega} ))^{p}
&\geq
\sum_{\ell=1}^L
\liminf_{n\to\infty} n s_n(\bHank( {\bomega}_{\ell}))^{p}.
\end{align*}
\end{theorem}

 Observe that formulas  \eqref{b8a} establish a one-to-one correspondence 
 between the unit circle $\bbT$ and the real axis $\bbR_{*}$ with the points $x=+\infty$ and $x=-\infty$ identified. 
 They yield also the one-to-one correspondence  
 between the singular supports of the symbols $\omega(\mu)$ and ${\bomega}(x)$ linked by equality  \eqref{eq:HG}. 
 Thus, Theorem~\ref{LOCALC} is a direct consequence of Theorem~\ref{LOCAL}.

\section{Applications of localization principle: discrete case}\label{sec.b}
 
\subsection{Discrete representation}\label{sec.a3}
For a   sequence $\{h(j)\}_{j=0}^\infty$     of complex numbers,
the Hankel operator $\Gamma (h)$ in the space $ \ell^2(\bbZ_+)$ is formally defined by the 
``infinite matrix'' $\{h(j+k)\}_{j,k=0}^\infty$:
\begin{equation}
(\Gamma (h) u) (j)= \sum_{k=0}^\infty h(j+k) u (k), \quad u= \{u(k)\}_{k=0}^\infty.
\label{eq:a5}
\end{equation}
The Hankel operators $\Gamma(h)$ in $\ell^2(\bbZ_+) $ and $\Hank(\omega)$ in $H^2(\bbT)$ are
related as follows. 
Let
$$
\calF: f\mapsto \{\wh f(j)\}_{j=0}^\infty,
\quad
\calF: H^2(\bbT) \to \ell^2 ({\bbZ}_{+}), 
 $$
be the discrete Fourier transform. 
Then the matrix elements of $H (\omega)$ in the orthonormal basis 
$\{\mu^j\}_{j=0}^\infty$ are
$$
(H(\omega) \mu^j, \mu^k)_{L^2(\bbT)}
=
\wh\omega(j+k), \quad j,k\geq0,
$$
so that 
\begin{equation}
\Gank(h) 
= 
{\mathcal F}H(\omega) {\mathcal F}^* \quad {\rm if} \quad  \wh \omega(j)
=
h(j), \quad j\geq0.
\label{b1}
\end{equation}
Since \eqref{b1} involves only the coefficients with $j\geq0$, for a given sequence $h$ the 
symbol  $\omega$ is not uniquely defined. 

\subsection{Plan of the approach} 

In our previous publication \cite{II} we considered compact \emph{self-adjoint} Hankel operators,
corresponding to sequences of real numbers of the type
\[
q(j)= j^{-1}(\log j)^{-\alpha}+ \text{error term}, \quad j\to\infty,
\label{z0a}
\]
where $\alpha>0$.  
Under the appropriate assumptions on the error term, we proved in \cite{II} that 
the positive   eigenvalues of the Hankel operator $\Gamma(q)$ have the asymptotics
$$
\lambda_n^{+}(\Gamma(q))= v(\alpha) n^{-\alpha} +o(n^{-\alpha}), 
\quad n\to\infty,
$$
where the coefficient $v(\alpha)$ is defined in \eqref{a7}. 
For negative eigenvalues, we have $\lambda_n^{-}(\Gamma(q))= o(n^{-\alpha})$ as $n\to\infty$.

In \cite{II} our analysis was based on the 
asymptotic form \eqref{z0a} and did not involve symbols directly.
In this paper, we check (this is an easy calculation, see Lemma~\ref{locs} below) that if $q(j)= j^{-1}(\log j)^{-\alpha}$, then a symbol $\sigma$ of $\Gank(q)$ can be chosen such that 
$\sing\supp \sigma=\{1\}$. 

Theorem~\ref{LOCAL} allows us to find the asymptotics of singular values 
for more general ``oscillating'' sequences of the type 
\[
h(j)=\sum_{\ell=1}^L b_\ell j^{-1}(\log j)^{-\alpha}\zeta_\ell^{-j} + \text{error term}, \quad j\to\infty,
\label{z9}
\]
where  
$\zeta_1,\dots,\zeta_L\in\bbT$ are distinct points and 
$b_1,\dots,b_L\in\bbC$ are arbitrary coefficients. 
It is easy to see  that the symbol corresponding to the $\ell$'th term in \eqref{z9} equals
$b_\ell \sigma (\mu/\zeta_\ell)$.
Hence its singular support  consists of one point $\{ \zeta_\ell\}$,
and so we are in the situation described by the localization principle for  $p=1/\alpha$.
The error term in  \eqref{z9} is treated by using the estimates from \cite{I} on singular values of Hankel operators.

Notice that the operators $\Gamma (h)$ corresponding to sequences $h$ of the class \eqref{z9} 
are in general not self-adjoint. 
We have information about the asymptotics of their singular values, but not of 
their eigenvalues.

\subsection{Main result in the discrete case}

In order to state our requirements 
 on the error term in \eqref{z9},  we need some notation.
Let   
\begin{equation}
M(\alpha)=
\begin{cases}
[\alpha]+1, & \text{ if } \alpha\geq1/2,
\\
0, & \text{ if } \alpha<1/2,
\end{cases}
\label{a3}
\end{equation}
where $[\alpha]$ is the integer part of $\alpha$. 
For a sequence $h=\{h(j)\}_{j=0}^\infty$, we define iteratively the sequences
$h^{(m)}=\{h^{(m)}(j)\}_{j=0}^\infty$, $m=0,1,2,\dots$, by setting $h^{(0)}(j)=h(j)$  
and 
\begin{equation}
h^{(m+1)}(j)=h^{(m)}(j+1)-h^{(m)}(j), 
\quad j\geq0.
\label{eq:It}\end{equation}
Note that if $h(j)=j^{-1}(\log j)^{-\alpha}$ for sufficiently large $j$, then for all $m\geq1$ the sequences $h^{(m)}$ satisfy 
\begin{equation}
h^{(m)}(j)=O(j^{-1-m}(\log j)^{-\alpha}), \quad j\to\infty.
\label{eq:It1}\end{equation}

Now we are in a position to state precisely our result on Hankel operators with matrix elements \eqref{z9}.

\begin{theorem}\label{thm.a4}
Let $\alpha>0$, let $\zeta_1, \dots,\zeta_L\in\bbT$ be distinct numbers, 
and let $b_{1},\dots,b_L\in\bbC$.
Let $h$ be a sequence of complex numbers such that
\[
h(j)=\sum_{\ell=1}^L \bigl(b_\ell j^{-1}(\log j)^{-\alpha}+ g_\ell(j)\bigr)\zeta_\ell^{-j},
\quad
j\geq2,
\label{a13}
\]
where the error terms $g_\ell$, $\ell=1,\ldots, L$, satisfy the estimates
\begin{equation}
g_\ell^{(m)}(j)
=o(j^{-1-m}(\log j)^{-\alpha}), 
\quad j\to\infty, 
\label{a5a}
\end{equation}
for all   $m=0,1,\dots,M(\alpha)$
$( M(\alpha)$ is given by \eqref{a3}$)$.
Then the singular values of the Hankel
  operator $\Gamma (h)$ defined in $\ell^2(\bbZ_+)$   by formula \eqref{eq:a5}  satisfy the asymptotic relation
  \[
s_n(\Gamma (h))= c \, n^{-\alpha}+o(n^{-\alpha}), 
\quad n\to\infty,
\label{z0c}
\]
where
\[
c=v(\alpha)\Big(\sum_{\ell=1}^L \abs{b_\ell}^{1/\alpha}\Big)^\alpha 
\label{a14}
\]
and the coefficient 
$v(\alpha)$ is given by formula \eqref{a7}.
\end{theorem}

This result means that asymptotically the singular value counting function of the operator $\Gamma (h)$ is the sum of such functions for every term in the right-hand side of \eqref{a13}.

\section{Proof of Theorem~\ref{thm.a4}}\label{sec.a}

\subsection{Singular value estimates and asymptotics}
We need two results obtained in our papers \cite{I, II}.
Let $M(\alpha)$ be as in \eqref{a3}.

\begin{theorem}\cite[Theorem~2.3]{I}\label{thm.z1}
Suppose that a sequence $g $ satisfies 
\begin{equation}
g^{(m)}(j)
=o(j^{-1-m}(\log j)^{-\alpha}), 
\quad j\to\infty, 
\label{a5b}
\end{equation}
for some $\alpha>0$
and for all $m=0,1,\dots,M(\alpha)$. Then  
\[
s_n(\Gank(g))=o(n^{-\alpha}), \quad n\to\infty.
\label{a5c}
\]
\end{theorem}

In \cite{I}
we also have a result with $O$ instead of $o$ in both 
\eqref{a5b} and  \eqref{a5c}, but we do not use it in this paper. 
Observe that for $\alpha<1/2$ we need only the estimate on $g$, 
whereas for $\alpha\geq1/2$ we also need estimates on the iterated differences 
$g^{(m)}$. 

\begin{theorem}\cite[Theorem~1.1]{II}\label{thm.a1a}
Let $\alpha>0$,  and let the ``model sequence''  $q$  be  defined by 
\[
q(j)=j^{-1}(\log j)^{-\alpha} 
\label{a16}
\]
for all sufficiently large $j$ 
\emph{(\emph{the values ${q}(j)$ for any finite number of $j$  are unimportant})}.
Then 
$$
s_n(\Gamma({q}))= v(\alpha )n^{-\alpha}+o(n^{-\alpha}), 
\quad 
n\to\infty.
$$
\end{theorem}

Of course, this result corresponds to a particular case of  
Theorem~\ref{thm.a4}  with $L=1$, $\zeta_1=1$, $b_1=1$.

\subsection{The model symbol}
In order to combine the contributions of different terms in \eqref{a13}, we use 
the localization principle (i.e. Theorem~\ref{LOCAL}). 
To that end, we have to identify the singular support of the symbol corresponding
to the model sequence \eqref{a16}; we suppose that  \eqref{a16} is true for all $j\geq 2$ and put $q(0)=q(1)=0$.
We need to find a function $\sigma$ such that its Fourier coefficients $\wh \sigma(j)=q(j)$ for $j\geq0$. 
Of course, the choice of $\sigma$ is not unique. We will choose $\sigma$ corresponding to the 
\emph{odd} extension of  the sequence $q(j)$ to the negative $j$.

\begin{lemma}\label{locs}
Let  $\alpha\geq 0$, and let ${q}$ be  given by \eqref{a16}; set
\[
\sigma(\mu)
=
\sum_{j=2}^\infty {q}(j)(\mu^j -\overline{\mu}^j), 
\quad \mu\in\bbT.
\label{eq:odd}
\]
Then
$\sigma\in L^\infty (\bbT)$ and $\sigma\in C^\infty (\bbT\setminus \{1\})$.
\end{lemma}

\begin{proof}
Note that for all $\mu\in\bbT$, the series  \eqref{eq:odd}  converges absolutely if $\alpha>1$ 
and conditionally if $\alpha\leq1$. 

First, we check that $\sigma\in L^\infty(\bbT)$.
We write $\mu=e^{i\theta}$, $\theta\in(-\pi,\pi]$.
For $\theta\not=0$, we set $N=[(2\abs{\theta})^{-1}]$
and  write $\sigma=\sigma_1+\sigma_2$, where
\[
\sigma_1(\mu)
=
\sum_{j=2}^N {q}(j)(\mu^j-\overline{\mu}^j),
\qquad
\sigma_2(\mu)
=
\sum_{j=N+1}^\infty {q}(j)(\mu^j-\overline{\mu}^j).
\label{eq:odd1}
\]
We consider these two functions separately.
Using the bounds $q(j)\leq (\log 2)^{-1} j^{-1}$ and
$$
\abs{\mu^j-\overline{\mu}^j}=2\abs{\sin(j\theta)}\leq 2j\abs{\theta},
$$
 we obtain the estimate
$$
\abs{\sigma_1(\mu)}
\leq
2\abs{\theta}
\sum_{j=2}^N j{q}(j)
\leq
2  (\log 2)^{-1}\abs{\theta}N
\leq
(\log 2)^{-1}.
$$

In order to estimate $\sigma_2$, let us use summation by parts:
\begin{multline}
(\mu-1)\sum_{j=N+1}^\infty {q}(j) \mu^j
=
\sum_{j=N+1}^\infty {q}(j) (\mu^{j+1}-\mu^j)
\\
=
-\sum_{j=N+1}^\infty q^{(1)}(j)\mu^{j+1}
-
{q}(N+1)\mu^{N+1}
\label{a12}
\end{multline}
where $q^{(1)}(j)$ is defined by \eqref{eq:It}.
By  \eqref{eq:It1}, we have ${q}^{(1)}(j)=O(j^{-2})$, $j\to\infty$, and hence
$$
\Abs{(\mu-1)\sum_{j=N+1}^\infty {q}(j) \mu^j }
\leq C_{1} \big( \sum_{j=N+1}^\infty j^{-2} + N^{-1}\big)\leq
C_{2} N^{-1}.
$$
In view of definition  \eqref{eq:odd1}, it follows that 
$$
\abs{\sigma_2(\mu)}\leq 2 \Abs{\sum_{j=N+1}^\infty {q}(j) \mu^j }
\leq  
\frac{2 C_{2}}{N\abs{\mu-1}}
=
\frac{ 2 C_{2}}{[(2\abs{\theta})^{-1}]\abs{e^{i\theta}-1}}
\leq 
C.
$$
Thus  $\sigma_2\in L^\infty(\bbT)$. 

It remains to prove   that
$\sigma\in C^M(\bbT\setminus\{1\})$ for any $M\in\bbN$. Choose $\mu\in\bbT$ and
put $a(j)=\mu^j$; then, by  definition \eqref{eq:It}, $a^{(M+1)}(j)=(\mu-1)^{M+1}\mu^j$. 
Similarly to \eqref{a12},
  by a repeated summation by parts procedure, we obtain the identity
\begin{multline}
(\mu-1)^{M+1}\sum_{j=2}^\infty {q}(j) \mu^j
=
\sum_{j=2}^\infty {q}(j) a^{(M+1)}(j)
\\
=
(-1)^{M+1}\sum_{j=2}^\infty {q}^{(M+1)}(j) a(j)+p_M(\mu)
\label{a15}
\end{multline}
with some polynomial $p_M$. Since, by \eqref{eq:It1}, 
$
q^{(M+1)}(j) = O(j^{-2-M})   
$ as $j\to\infty$ and $ a(j)=\mu^j$,
the function of $\mu$ in the right-hand side of \eqref{a15} is in $C^M(\bbT)$. 
It follows that $\sigma\in C^M(\bbT\setminus\{1\})$ and hence   $\sigma\in C^\infty(\bbT\setminus\{1\})$. 
\end{proof}

\begin{remark}
(i)  Of course,  the singular support of 
$\sigma$ is non-empty, that is,  $1\in\sing\supp \sigma$. 
In fact, it can be verified  that  
$$
 \sigma(e^{i\theta})=\pi i \sign\theta \abs{\log\abs{\theta}}^{-\alpha}(1+o(1)), \quad \theta\to0.
$$

(ii)   If $\alpha>1$, then instead of the 
odd extension of $q(j)$ to the negative $j$, one can extend it by zero, i.e.
one can choose
$$
\wt\sigma(\mu)
=
\sum_{j=2}^\infty {q}(j)\mu^j.
$$
This doesn't work for $\alpha\leq1$ since $\wt\sigma(\mu)$ is unbounded as $\mu\to 1$ in this case.
\end{remark}

According to  definition \eqref{eq:odd}, we have
$\wh{\sigma} (j)=q(j)$  for all $j\geq 0$. Hence, it follows from relation 
 \eqref{b1} that the operators $H(\sigma)$ and $\Gamma(q)$ are unitarily equivalent. So the next assertion is a direct consequence of Theorem~\ref{thm.a1a}.

\begin{theorem}\label{mod}
Let the function $\sigma(\mu)$ be defined by formula \eqref{eq:odd} where $q(j)$ are given by \eqref{a16} and $\alpha>0$.
Then the following asymptotic relation holds true:
$$
s_n(H(\sigma))= v(\alpha )n^{-\alpha}+o(n^{-\alpha}), 
\quad 
n\to\infty.
$$
\end{theorem}

\subsection{Rotation of the symbol}
For a parameter $\zeta\in\bbT$, 
let $R_{\zeta}$ be the  ``rotation by $\zeta$'' operator:  
$$
(R_{\zeta}f)(\mu)= f(\mu/\zeta).
$$
Obviously, $R_\zeta$ is a
unitary operator  in $L^2 (\bbT)$ and in  $H^2 (\bbT)$.
Similarly, let
$V_\zeta$ be the  multiplication by $\zeta^{-j}$:
$$
(V_\zeta u)(j)=\zeta^{-j}u(j) .
$$
 Obviously, $V_\zeta$ is a unitary operator  in $\ell^2 (\bbZ)$ and in $\ell^2 (\bbZ_+)$.

\begin{lemma}\label{zeta} 
For  arbitrary $\zeta\in\bbT$, we have  the following statements:
\begin{enumerate}[\rm(i)]
\item
If $\omega\in L^\infty (\bbT)$, then
$$
H(R_\zeta\omega)=R_{\zeta} H(\omega ) R_{\zeta}.
$$
In particular, if $H(\omega)$ is compact, then
$$
s_n(\Hank(R_\zeta \omega))=s_n(\Hank(\omega)), \quad \forall n\geq1.
$$

\item
 For any sequence $h$ such that $\Gank(h)$ is bounded, we have
$$
\Gank(V_\zeta h)=V_\zeta \Gank(h) V_\zeta.
$$
In particular,  if $\Gank(h)$ is compact, then
$$s_n(\Gank(V_\zeta h))=s_n(\Gank(h)), \quad\forall n\geq1.$$
\end{enumerate}
\end{lemma}
 
\begin{proof}
Since
$$
P_{+}R_{\zeta}=R_{\zeta} P_{+} \quad {\rm and} \quad R_{\zeta}\inv R_{\zeta}= \inv,
$$
 assertion (i) is a direct consequence of the definition 
\eqref{eq:HH}  of the Hankel operator $H(\omega ) $. Assertion (ii) immediately follows from definition \eqref{eq:a5}. 
 \end{proof}

\subsection{Putting things together.}
Let the symbol $\sigma(\mu)$ be defined by relation \eqref{eq:odd} and let
\[
\omega_\natural (\mu)=\sum_{\ell=1}^L \omega_\ell (\mu)\quad {\rm where} \quad \omega_\ell (\mu)=b_\ell \sigma (\mu/\zeta_{\ell}).
\label{eq:SYM}
\]
According to Theorem~\ref{mod} and Lemma~\ref{zeta}(i)  we have
$$
s_n(H(\omega_\ell))= | b_{\ell} | v(\alpha )n^{-\alpha}+o(n^{-\alpha}), 
\quad 
n\to\infty.
$$
It follows from Lemma~\ref{locs} that $\omega_{\ell}\in L^\infty(\bbT )$ and $\omega_{\ell}\in C^\infty(\bbT\setminus\zeta_{\ell})$. Since  $\zeta_1,\dots,\zeta_L$ are distinct points, the localisation principle (Theorem~\ref{LOCAL}) is applicable to the sum \eqref{eq:SYM}.
 This yields
\[
\lim_{n\to\infty}ns_n(\Hank(\omega_\natural ))^{p}
=
\sum_{\ell=1}^L
\lim_{n\to\infty}ns_n(\Hank( \omega_\ell ))^{p}
=
v(\alpha)^{p}
\sum_{\ell=1}^L\abs{b_\ell}^{p}, \quad p=1/\alpha.
\label{a22}
\]

Note that, by the definition \eqref{eq:SYM},
$$
\wh\omega_\ell (j)= b_{\ell}\zeta^{-j}  \wh\sigma (j)    
$$
and hence according to formula  \eqref{eq:odd} 
$$
\wh\omega_\natural  (j)=\sum_{\ell=1}^L b_\ell \zeta_{\ell}^{-j} j^{-1} (\log j)^{-\alpha} =: h_\natural (j), \quad j\geq 2.
$$
Set $h_{\natural}(0)=h_{\natural}(1)=0$. 
Since the operators $\Hank(\omega_\natural)$ and $\Gank(h_\natural )$ are unitarily equivalent,  it follows from \eqref{a22} 
that
\[
\lim_{n\to\infty}ns_n(\Gank(h_\natural ))^{p}
=
v(\alpha)^{p}\sum_{\ell=1}^L\abs{b_\ell}^{p}.
\label{eq:22}
\]

Next, we consider the error term
$$
g(j)=\sum_{\ell=1}^L     \zeta_\ell^{-j} g_\ell(j)
$$
in \eqref{a13}. According to condition \eqref{a5a} it follows from Theorem~\ref{thm.z1} that $s_{n}(\Gamma (g_\ell))=o(n^{-\alpha})$ as $n\to\infty$. By Lemma~\ref{zeta}(ii), we also have $s_{n}(\Gamma (V_{\zeta_\ell} g_\ell))=o(n^{-\alpha})$ and hence 
\[
s_{n}(\Gamma ( g)) =o(n^{-\alpha})\quad {\rm as } \quad n\to\infty.
\label{eq:ERR}
\]

Since
$$
\Gank(h)=\Gank(h_\natural)+\Gank(g),
$$
we can use Lemma~\ref{lma.a2} with 
$A=\Gank(h_\natural )$ and 
$B=\Gank(g)$.  The required relations \eqref{z0c}, \eqref{a14} follow from \eqref{eq:22} and \eqref{eq:ERR}. 
   \qed

\section{Applications of localization principle: continuous case}\label{sec.c}

\subsection{Hankel operators in $L^2(\bbR_+)$}

Integral Hankel operators $\bGank(\bh)$ in the space  $L^2(\bbR_+)$
  are  defined by the relation
\begin{equation}
(\bGank(\bh)\bu)(t)=\int_0^\infty \bh(t+s)\bu (s)ds, \quad {\bf u}\in C_{0}^\infty (\bbR_{+}),
\label{a0}
\end{equation}
where at least $\bh\in L^1_\loc(\bbR_+)$; this function is called the \emph{kernel} of the Hankel operator $\bGank(\bh)$. 
Under the assumptions on $\bh$ below the operators $\bGank(\bh)$ are compact.

Similarly to the discrete case, Hankel operators $\bHank( {\bomega} )$ in the Hardy space $H^2_+(\bbR)$ 
are  unitarily equivalent to integral operators $\bGank(\bh)$ in the space $L^2(\bbR_{+})$:
\begin{equation}
\Phi \bHank( {\bomega} )\Phi^* = \bGank(\bh)
\quad \text{ if }\quad
\bh(t)=\frac1{\sqrt{2\pi}}  \wh{\bomega} (t)
\quad \text{for $t>0$.}
\label{e4}
\end{equation}
The Fourier transform $\wh{\bomega}$ of $ \bomega\in L^\infty (\bbR)$ should in general be understood 
in the sense of distributions (for example, on the Schwartz class ${\mathcal S}' (\bbR)$) and 
the precise meaning of \eqref{e4} is given by the equation
$$
(\bHank({\bomega})\Phi^* {\bf u} , \Phi^* {\bf u}) 
=
( \bGank(\bh) {\bf u},{\bf u}), \quad {\bf u}\in C_{0}^\infty (\bbR_{+}).
$$
A function $\bomega (x)$ satisfying  \eqref{e4} is known as a symbol 
of the Hankel operator $\bGank(\bh)$.

\subsection{Main result in the continuous case}

In the discrete case, the spectral asymptotics of $\Gank(h)$ is determined by the asymptotic
behavior of the sequence $h(j)$ as $j\to\infty$. 
In the continuous case, the behavior of the kernel $\bh(t)$ for $t\to\infty$ and for $t\to 0$
as well as  the local singularities of $\bh(t)$  at positive points $t$
contribute to the spectral asymptotics of $\bGank(\bh)$.
In the following   result we exclude  local singularities. We denote $\jap{x}=\sqrt{1+\abs{x}^2}$.

\begin{theorem}\label{thm.d3}
Let $\alpha>0$, let $a_1,\dots,a_L\in\bbR$ 
be distinct numbers and let $\bb_0,\bb_1,\dots,\bb_L\in\bbC$. 
Let the number $M=M(\alpha)$ be given by \eqref{a3}. 
Suppose that $\bh\in L^\infty_\loc(\bbR_+)$ if $\alpha<1/2$ and $\bh\in C^M(\bbR_+)$ if $\alpha\geq1/2$. 
Assume that 
\begin{align}
\bh(t)&=\sum_{\ell=1}^L \bigl(\bb_\ell t^{-1}(\log t)^{-\alpha}+  \bg_\ell(t)\bigr)e^{- i a_\ell t}, 
\quad t\geq2,
\label{z2}
\\
\bh(t)&=\bb_0t^{-1}\bigl(\log(1/t)\bigr)^{-\alpha}+  \bg_0(t), 
\quad t\leq 1/2,
\label{z3}
\end{align}
where the error terms $\bg_\ell$ satisfy the estimates
\[
\bg_\ell^{(m)}(t)
=
o( t^{-1-m}\jap{\log t}^{-\alpha}), 
\quad 
m=0,\dots,M(\alpha),  
\label{z4}
\]
as $ t\to\infty$ for $\ell=1,\ldots, L$ and
 as $t\to 0$ for $\ell=0$.
Then the singular values of the integral Hankel
operator $\bGank(\bh)$ in $L^2(\bbR_+)$  satisfy the asymptotic relation
\[
s_n(\bGank(\bh))= {\bf c} \, n^{-\alpha}+o(n^{-\alpha}), \quad n\to\infty,
\label{d11}
\]
where 
\[
{\bf c}  =v(\alpha)\Big(\sum_{\ell=0}^L \abs{\bb_\ell}^{1/\alpha}\Big)^\alpha
\label{eq:d11}
\]
and the coefficient 
$v(\alpha)$ is given by formula \eqref{a7}.
\end{theorem}

The proof in the continuous case follows the same general outline
as in the discrete case
with the only difference that the singularity of the kernel $\bh(t)$ at 
$t=0$ has to be treated separately. It corresponds to the singularity
of the symbol $\bomega(x)$ at infinity.

In Section~\ref{sec.f} we consider kernels $\bh (t)$ that   have a singularity at some positive point and admit representation \eqref{z2} for large $t$. It turns out that, similarly to  Theorem~\ref{thm.d3}, the contributions of the singularities of these two types  to the asymptotics 
of singular values are independent of each other.

\section{Proof of Theorem~\ref{thm.d3}}\label{sec.d}
 
The proof of Theorem~\ref{thm.d3} follows the scheme of the proof of Theorem~\ref{thm.a4}. 
The only new  point is that now we have to additionally establish the correspondence 
between symbols singular at infinity and kernels singular at $t=0$.

\subsection{Singular value estimates and asymptotics}

Let us state  the analogues of Theorems~\ref{thm.z1} and \ref{thm.a1a}. 

\begin{theorem}\label{thm.d1}\cite[Theorem~2.8]{I}
Let $\alpha>0$, and let the number $M=M(\alpha)$ be given by \eqref{a3}. Suppose that
$\bg \in L^\infty_\loc(\bbR_+)$ if $\alpha<1/2$ and $\bg \in C^M(\bbR_+)$ if $\alpha\geq1/2$.
Assume   that  
\begin{equation}
 \bg^{(m)}(t)
=o(t^{-1-m}\jap{\log t}^{-\alpha}) 
\quad \text{ as $t\to0$ and as $t\to\infty$}
\label{d3}
\end{equation}
for all $m=0,1,\dots,M$. Then   
\[
s_n(\bGank( \bg))=o(n^{-\alpha}), 
\quad 
n\to\infty.
\label{d4}
\]
\end{theorem}

In \cite{I} we also have a result with $O$ instead of $o$ in \eqref{d3} and \eqref{d4}, 
although we will not need it in this paper.  
Observe that for $\alpha<1/2$ we need only the estimate on $\bg$, 
whereas for $\alpha\geq1/2$ we also need estimates on the derivatives
$\bg^{(m)}$.

Next,  we define  model kernels $\bq_0$, $\bq_\infty$.
Choose some non-negative functions $\chi_0, \chi_\infty \in C^\infty(\bbR)$  such that
$$
\chi_0(x)=
\begin{cases}
1& \text{for $\abs{x}\leq c_1$,}
\\
0& \text{for  $\abs{x}\geq c_2$,}
\end{cases}
\quad
\chi_\infty(x)=
\begin{cases}
0& \text{for $\abs{x}\leq C_1$,}
\\
1& \text{for  $\abs{x}\geq C_2$,}
\end{cases}
$$
for some $0<c_1<c_2<1$ and $1<C_1<C_2$.

\begin{theorem}\cite[Theorem~1.2]{II}\label{thm.a1aC}
For $\alpha>0$, set
\[
\bq_0(t)=\chi_0(t)t^{-1}(\log(1/t))^{-\alpha}, 
\quad
\bq_\infty(t)=\chi_\infty(t)t^{-1}(\log t)^{-\alpha}, 
\quad
t>0.
\label{eq:qq}
\]   
Then 
$$
s_n(\bGank(\bq_0))= v(\alpha) \, n^{-\alpha}+o(n^{-\alpha}), \quad 
s_n(\bGank(\bq_\infty))= v(\alpha) \, n^{-\alpha}+o(n^{-\alpha}), \quad 
n\to\infty.
$$
\end{theorem}

Of course, 
this result corresponds to particular cases of  Theorem~\ref{thm.d3} 
with $L=1$, $a_{1}=0$, $\bb_{0} =1$, $\bb_1=0$ and $\bb_0=0$, $\bb_1=1$.

\subsection{Model symbols}
In order to put together the contributions of different terms in \eqref{z2} and \eqref{z3}, 
we use the localization principle in the form of Theorem~\ref{LOCALC}. 
To that end, we need to determine the singular supports of the symbols corresponding to the 
model kernels $\bq_0$, $\bq_\infty$.
Again, we will choose functions $\bsigma_0$, $\bsigma_\infty$ whose Fourier transform 
coincides with the odd extension of $\bq_0$, $\bq_\infty$ to the real line. 
The proof below is very similar to that of Lemma~\ref{locs}.

\begin{lemma}\label{locsC}
Let $\bsigma_0$, $\bsigma_\infty$ be defined by 
\[
\bsigma_0(x)
=
2i
\int_0^\infty \bq_0(t)\sin(xt)dt, 
\quad
\bsigma_\infty(x)
=
2i
\int_0^\infty \bq_\infty(t)\sin(xt)dt,
\quad x\in\bbR,
\label{eq:int}
\]
where $\bq_0(t)$ and $\bq_\infty(t)$ are given by \eqref{eq:qq} with $\alpha\geq 0$.
Then 
$\bsigma_0,\bsigma_\infty\in L^\infty(\bbR)$ and $\bsigma_{0}\in C^\infty (\bbR)$, $\bsigma_\infty\in C^\infty (\bbR_{*}\setminus\{0\})$. 
\end{lemma}

\begin{proof}
Note that for all $x\in\bbR$, the first integral in   \eqref{eq:int}  converges absolutely while the second one 
converges absolutely
for $\alpha>1$ 
and conditionally for $\alpha\leq1$. 

Since the integral in the definition \eqref{eq:int} of $\bsigma_0$ is taken over a finite interval, 
we can differentiate this integral with respect to $x$ arbitrary many times. 
Hence $\bsigma_0\in C^\infty(\bbR)$. 
To prove that $\bsigma_\infty\in C^\infty (\bbR_{*}\setminus\{0\})$, we  
integrate by parts $2M+2$ times in the definition \eqref{eq:int}:
$$
\bsigma_\infty(x)= 2i
(-1)^{M+1} x^{-2M-2}
\int_{0}^\infty \bq_\infty^{(2M+2)}(t)\sin(xt)dt.
$$
Since $\bq_\infty^{(2M+2)}(t)=O(\abs{t}^{-2M-3})$ as $\abs{t}\to\infty$, 
we see that $\bsigma_\infty\in C^{m}(\bbR\setminus\{0\})$ and
$\bsigma^{(m)}_\infty(x)\to0$ for $m=0,1,\ldots, 2M+1$ as $\abs{x}\to\infty$. 
Finally, we use that $M$ is arbitrary.

It remains to prove that the functions $\bsigma_0$ and $\bsigma_\infty $ are bounded. Below $\kappa=0$ or $\kappa=\infty$. We may suppose that $x>0$. Write $\bsigma_\kappa=\bsigma_{\kappa}^{(1)}+\bsigma_{\kappa}^{(2)}$, where
$$
\bsigma_{\kappa}^{(1)}(x)
=
2i
\int_0^{1/ x}\bq_\kappa(t)\sin(xt)dt,
\quad
\bsigma_{\kappa}^{(2)}(x)
=
2i
\int_{1/ x}^\infty \bq_\kappa(t)\sin(xt)dt.
$$
Since $\abs{\sin(xt)}\leq xt$, for both functions $\bsigma_{\kappa}^{(1)}$ we have the estimate
$$
\abs{\bsigma_{\kappa}^{(1)}(x)}
\leq 
2 x\int_0^{1/ x}\bq_\kappa(t)tdt
\leq 
C
$$
because $\bq_\kappa(t)t $ are bounded functions.
For $\bsigma_{\kappa}^{(2)}$, integrating by parts once, we get 
$$
\bsigma_{\kappa}^{(2)}(x)
=
-\frac{2i}{x}
\int_{1/x}^\infty \bq_\kappa (t)(\cos(xt))'dt
=
\frac{2i}{x}
\bq_\kappa(1/x)\cos 1
+
\frac{2i}{x}
\int_{1/x}^\infty \bq_\kappa '(t) \cos(xt)dt.
$$
The first term in the right-hand side is bounded because
$\bq_\kappa (t)t$ are bounded functions. 
The second term is also bounded  because the functions
$ \bq_\kappa '(t) t^{2}$ are bounded. 
\end{proof}

\begin{remark}
(i)  Of course the singular supports of 
$\bsigma_{0}$ and $\bsigma_{\infty}$ are non-empty sets in $\bbR_{*}$, that is,  $  \sing\supp \bsigma_{0} = \{\infty\}$ and  $  \sing\supp \bsigma_{\infty} =\{0\}$. 
In fact, it can be verified that the symbols $\bsigma_0$, $\bsigma_\infty$ satisfy 
the asymptotics
\begin{align*}
\bsigma_0(x)&=\pi i\sign x\abs{\log\abs{x}}^{-\alpha}(1+o(1)), \quad x\to\infty,
\\
\bsigma_\infty(x)&=\pi i \sign x\abs{\log\abs{x}}^{-\alpha}(1+o(1)), \quad x\to0.
\end{align*}

(ii)   
For some values of $\alpha$, instead of the
odd extension of $\bq_0(t)$ and $\bq_\infty(t)$ to the negative $t$, 
one can extend them by zero, i.e. one can choose
\begin{align*}
\wt\bsigma_0(x)
&=
\int_0^\infty \bq_0(t)e^{ixt}dt, \quad \text{ if $\alpha<1$,}
\\
\wt\bsigma_\infty(x)
&=
 \int_0^\infty \bq_\infty(t)e^{ixt}dt, \quad \text{ if $\alpha>1$,}
\end{align*}
instead of $\bsigma_0(x)$, $\bsigma_\infty(x)$, respectively.
\end{remark}

Recall that the Hankel operators  in the Hardy space $H_{+}^2(\bbR)$ were  defined by formula \eqref{e3}. 
The next assertion plays the role of  Theorem~\ref{mod}.

\begin{theorem}\label{modco}
Let the functions $\bsigma_{0}$ and $\bsigma_\infty$ be defined by formulas \eqref{eq:int} where $\bq_{0}(t)$ and $\bq_\infty(t)$ are given by \eqref{eq:qq} and $\alpha>0$. 
Then the following asymptotic relations hold true:
$$
s_n(\bHank(\bsigma_{0}))= v(\alpha )n^{-\alpha}+o(n^{-\alpha}), \quad s_n(\bHank(\bsigma_{\infty}))
= v(\alpha )n^{-\alpha}+o(n^{-\alpha}),
\quad 
n\to\infty.
$$
\end{theorem}

\begin{proof}
Observe that
 \begin{equation}
\frac1{\sqrt{2\pi}}\wh\bsigma_0(t)=\bq_0(t),
\quad
\frac1{\sqrt{2\pi}}\wh\bsigma_\infty(t)=\bq_\infty(t),  
\quad t>0,
\label{eq:UE1}
\end{equation}
where the Fourier transform is understood  in the class of distributions $\mathcal{S}(\bbR)'$. 
Indeed, the second equality 
\eqref{eq:UE1} follows directly from definition  \eqref{eq:int} because $\bq_\infty\in \mathcal{S}(\bbR)'$. 
In order to prove the first equality 
in \eqref{eq:UE1}, we have to take into account that for $\alpha\leq 1$ the function $\bq_0(t)$ is not integrable in a neighborhood of the point $t=0$. Therefore we first extend $\bq_0$ by the formula
$$
\langle \bq_{0}^{(\ext)}, \varphi\rangle=\int_{0}^\infty \bq_{0}(t) (\bar{\varphi}(t)-\bar{\varphi}(-t))dt
$$
to the distribution $\bq_{0}^{(\ext)} \in \mathcal{S}(\bbR)'$. 
According to the first formula in  \eqref{eq:int}, 
the function $(2\pi)^{-1/2}{\bsigma_{0}}$ is the Fourier transform   of 
$\bq_{0}^{(\ext)}$.  
Thus $(2\pi)^{-1/2}\wh\bsigma_0(t)=\bq_{0}^{(\ext)}(t)$, which coincides with 
the first relation in \eqref{eq:UE1} for $t>0$.

In view of relation \eqref{e4}, it follows from \eqref{eq:UE1} that
$$
\Phi \bHank( {\bsigma_{0}} )\Phi^* = \bGank(\bq_{0})
\quad \text{and}\quad
\Phi \bHank( {\bsigma_{\infty}} )\Phi^* = \bGank(\bq_{\infty}).
$$
Therefore we only have to use  Theorem~\ref{thm.a1aC} to complete the proof. 
\end{proof}
 
\subsection{Shifts of symbols}

For a parameter $a\in\bbR$, let $\bR_a $ be the shift
$$
(\bR_a {\bf f})(x)= {\bf f}(x-a).
$$
Obviously, $\bR_a$ is a
unitary operator in $L^2 (\bbR)$ and
$H_+^2 (\bbR)$. 
Of course,   now $\bR_a$ is  not  a rotation, 
but we keep the letter $\bR$ in order to maintain the analogy between the 
discrete and continuous cases.

Similarly, let $\bV_a $ be the multiplication operator 
$$
(\bV_a {\bf u})(t)=e^{- iat} {\bf u} (t), \quad t> 0.
$$
Obviously, $\bV_a$ is a unitary operator  in $L^2 (\bbR)$ and in $L^2 (\bbR_+)$.

\begin{lemma}\label{zetaC} 
For arbitrary  $a\in\bbR$, we have the following statements: 
\begin{enumerate}[\rm(i)]
\item
For any $\bomega\in L^\infty (\bbR)$, we have
$$
\bHank(\bR_a\bomega)=\bR_a \bHank(\bomega) \bR_a.  
$$
In particular,  if $\bHank(\bomega)$ is compact, then
$$
s_n(\bHank(\bR_a\bomega))=s_n(\bHank(\bomega)), \quad \forall n\geq1.
$$

\item
Suppose that $\bGank({\bf h})$ is bounded; then 
$$
\bGank(\bV_a\bh)=\bV_a \bGank(\bh) \bV_a .
$$
In particular,  if $\bGank(\bh)$ is compact, then
$$
s_n(\bGank(\bV_a\bh))=s_n(\bGank(\bh)), \quad \forall n\geq1.
$$
\end{enumerate}
\end{lemma}

 \begin{proof}
 Since
$$
\bP_{+}\bR_a=\bR_a \bP_{+} \quad {\rm and} \quad \bR_a {\binv} \bR_a= {\binv},
$$
the first assertion is a direct consequence of the definition 
\eqref{e3}  of the Hankel operator $\bHank(\bomega) $ in $H^2(\bbR)$. The second assertion immediately follows from the definition \eqref{a0}. 
 \end{proof}

\subsection{Putting things together}
Let the symbols $\bsigma_{0}(x)$ and $\bsigma_\infty(x)$ be defined by relations \eqref{eq:int} and let
\[
\bomega_{\natural}(x)= \bomega_0(x) +\sum_{\ell=1}^L \bomega_\ell (x),\quad {\rm where}
 \quad \bomega_0 (x)=\bb_0 \bsigma_{0} (x ),  \quad \bomega_\ell (x)=\bb_\ell \bsigma_{\infty} (x-a_{\ell}).
\label{eq:SYMC}
\]
According to Theorem~\ref{modco} and Lemma~\ref{zetaC}(i)  we have
$$
s_n(\bHank(\bomega_\ell))=  | \bb_{\ell} | v(\alpha )n^{-\alpha}+o(n^{-\alpha}),  \quad 
n\to\infty,
$$
for all $\ell=0,1,\ldots,L$.
It follows from Lemma~\ref{locsC} that $\bomega_{\ell}\in L^\infty(\bbR )$ for all $\ell=0,1,\ldots,L$, $\bomega_0\in C^\infty(\bbR)$  and $\bomega_{\ell}\in C^\infty(\bbR_{*}\setminus a_{\ell})$ for   $\ell=1,\ldots,L$. Since  $a_1,\dots, a_L$ are distinct points, the localisation principle (Theorem~\ref{LOCALC}) is applicable to the sum \eqref{eq:SYMC}.
 This yields
\[
\lim_{n\to\infty}ns_n(\Hank(\bomega_\natural))^{p}
=
\sum_{\ell=0}^L
\lim_{n\to\infty}ns_n(\Hank( \bomega_\ell ))^{p}
=
v(\alpha)^{p}
\sum_{\ell=0}^L\abs{\bb_\ell}^{p},\quad p=1/\alpha.
\label{eq:a22C}
\]
By  definition \eqref{eq:SYMC}, we have

$$
\wh\bomega_0 (t)= \bb_{0}\wh\bsigma_0 (t)
\quad\text{ and }\quad
\wh\bomega_\ell (t)= \bb_{\ell}\wh\bsigma_{\infty} (t) e^{-i a_{\ell } t} , \quad \ell =1,\ldots, L.
$$
Therefore,  according to formulas  \eqref{eq:qq}  and \eqref{eq:UE1}, we have
$$
\frac{1}{\sqrt{2\pi}}\wh\bomega_{\natural}(t)=\bb_0 \chi_0 (t) t^{-1} |\log t|^{-\alpha} + \sum_{\ell=1}^L \bb_\ell \chi_{\infty} (t) t^{-1} |\log t|^{-\alpha}  e^{-i a_{\ell } t} =: \bh_{\natural} (t), \quad t> 0.
$$
In view of relation  \eqref{e4} it now follows from \eqref{eq:a22C} 
that
$$
\lim_{n\to\infty}ns_n(\bGank(\bh_{\natural}))^{p}
=
v(\alpha)^{p}\sum_{\ell=0}^L\abs{\bb_\ell}^{p}.
$$
Next, we consider the error term
$$
\bg(t)= \bh(t)-\bh_{\natural}(t)=\bg_0(t) +\sum_{\ell=1}^L   \bg_\ell(t) e^{-i a_{\ell } t}
$$
where all functions $ \bg_\ell(t) $, $\ell=0,1,\ldots,L$,  satisfy the condition  \eqref{z4} both  for $t\to 0$ and $ t\to\infty$.
It  follows from Theorem~\ref{thm.d1} and Lemma~\ref{zetaC}(ii) that $s_{n}(\bHank (\bg_\ell))=o(n^{-\alpha})$   and hence 
\[
s_{n}(\bHank ( \bg)) =o(n^{-\alpha})\quad {\rm as } \quad n\to\infty.
\label{eq:ERRC}
\]
Since
$$
\bHank(h)=\bHank(\bh_{\natural})+\bHank(\bg),
$$
we can use Lemma~\ref{lma.a2} with 
$A=\bHank(\bh_{\natural})$ and 
$B=\bHank(\bg)$.  The required relations \eqref{d11}, \eqref{eq:d11} follow from \eqref{eq:a22C} and \eqref{eq:ERRC}. 
\qed

\section{Local singularities of the kernel}\label{sec.f}

The localization principle shows that the results on the asymptotics of singular values of different Hankel operators can be combined provided that the singular supports  of their symbols are disjoint.  
This idea has already been illustrated by Theorems~\ref{thm.a4} and \ref{thm.d3}. 
Here we apply the same arguments to  kernels $\bh (t)$ satisfying condition \eqref{z2} as $t\to\infty$ and    singular at some  point $t_0>0$. Below $\1_+ (t)$ is the characteristic function of $\bbR_{+}$.

The effect of local singularities of $\bh (t)$ on the asymptotics 
of singular values of the corresponding Hankel operator $\bGank(\bh)$ 
was studied in \cite{Part} and later in \cite{Yafaev2}. The methods of these papers are quite different. 
We use the following result obtained in \cite{Yafaev2}.

\begin{lemma}\label{sing1}
Let $t_{0}>0$, $m\in {\bbZ}_{+}$ and  
\begin{equation}
{\bf a}_{m} (t)=  (t_0-t)^{m} \1_+ (t_{0}-t ) . 
\label{eq:sing1}\end{equation}
Then  $\Ker \bGank({\bf a}_{m}) =L^2 (t_0, \infty)$ and
$$
\bGank({\bf a}_{m})\big|_{L^2 (0, t_0)}= m! A_{m}^{-1} 
$$
where the self-adjoint operator $A_{m}$ in $L^2 (0,t_0)$ is defined by the differential expression
$$
 (A_{m} {\bf u})(t)= (-1)^{m+1}  {\bf u}^{(m+1)}(t_0-t)
$$
and the boundary conditions 
 \begin{equation}
 {\bf u} (t_0)=\cdots = {\bf u}^{(m)}(t_0)=0.
\label{eq:sing4}\end{equation}
\end{lemma}

Note that the operator $A_{m}^2$   is given  by the differential expression
$$
(A_{m}^2 {\bf u})(t)= (-1)^{m+1}  {\bf u}^{(2m+2)}(t)
$$
and the boundary conditions \eqref{eq:sing4} and
 $$
 {\bf u}^{(m+1)} (0)=\cdots = {\bf u}^{(2m+1)}(0)=0.
$$
Thus $A_{m}^2$   is a regular differential operator and the asymptotics of its eigenvalues is given by the Weyl formula.
Therefore the following result is an immediate consequence of Lemma~\ref{sing1}. 

\begin{corollary}
Let the function ${\bf a}_{m} (t)$ be given by formula \eqref{eq:sing1}. Then
 \begin{equation}
s_{n} (\bGank({\bf a}_{m}))= m! t_{0}^{m+1}(\pi   n)^{-m-1} \big(1+ O (n^{-1})\big), \quad n\to \infty.
\label{eq:sing7}\end{equation}
\end{corollary}

Notice that formula \eqref{eq:sing7} was obtained much earlier in \cite{Part}
by a completely different method. 

We also note the explicit formula for the symbol $ {\boldsymbol{\tau}}_{m}(x)$  of the operator $\bGank({\bf a}_{m})$:
 \begin{equation}
 {\boldsymbol{\tau}}_{m}(x)= m! (ix)^{-m-1} \big( e^{i t_0 x } -\sum_{k=0}^m \frac{1} {k!} (i t_{0} x)^k
\big), \quad x\in\bbR.
\label{eq:sing6}
\end{equation}
Obviously, $ {\boldsymbol{\tau}}_{m}\in C^\infty (\bbR)$ and $ {\boldsymbol{\tau}}_{m}(x)$ is an oscillating function as $|x|\to\infty$.

We are now in a position to consider the general case.

\begin{theorem}\label{sing}
Let $t_{0}>0$, $m\in {\bbZ}_{+}$ and $\beta\in {\bbC}$.  Set
$$
\bh_{m} (t)= {\tt b} (t_{0}-t)^{m} \1_+ (t_{0}-t ) + \bh (t)
$$
where $\bh (t)$ satisfies the assumptions of Theorem~$\ref{thm.d3} $ with $\bb_{0}=0$ and $\alpha=m+1$. 
Then the singular values  of the operator $ \bGank(\bh_{m})$ satisfy the asymptotic 
\[
s_{n} (\bGank(\bh_{m}))= {\bf c}_{m} n^{-m-1} + o(n^{-m-1})
\label{eq:sloc}
\]
with
$$
{\bf c}_{m}= \Big( \pi^{-1}  t_0 (m! |{\tt b}  |)^{1/\alpha }+v (\alpha)^{1/\alpha} \sum_{\ell=1}^L \abs{b_\ell}^{1/\alpha}\Big)^\alpha,
\quad \alpha=m+1,
$$
and  
$v (\alpha)$ defined by \eqref{a7}.
\end{theorem}

\begin{proof}
 It is almost the same as that of Theorem~\ref{thm.d3}. Let us use  notation \eqref{eq:sing1}. 
 The asymptotics of the singular values of the operator  $ \bGank(  {\bf a}_{m})$  is given by formula \eqref{eq:sing7}.
 The operator $ \bGank(\bh )$ satisfies the assumptions of Theorem~\ref{thm.d3} so that the asymptotics of its eigenvalues is given by formula \eqref{d11}.  The symbol \eqref{eq:sing6} of the operator  $ \bGank(  {\bf a}_{m})$ is singular only at infinity. Neglecting the terms satisfying the assumptions of   Theorem~\ref{thm.d1} and using Lemma~\ref{locsC}, we see that the singular support of the symbol of the operator $ \bGank(\bh )$ consists of the points $a_{1}, \ldots, a_{L} \in \bbR$. Therefore applying Theorem~\ref{LOCALC}, we conclude the proof.
\end{proof}

\begin{remark}\label{sire}
We have chosen $\alpha=m+1$ in Theorem~\ref{sing} since in this case both the local singularity and the ``tail" of $h(t)$ at infinity contribute to the asymptotic coefficient ${\bf c}_{m} $ in \eqref{eq:sloc}.
\end{remark}

 Observe that we have excluded the term  \eqref{z3} singular at $t=0$ in Theorem~\ref{sing}  because the corresponding symbol is singular
 at the same point $x=\infty$ as the function \eqref{eq:sing6}. In this case one might expect that the   contributions of singularities of $\bh (t)$ at $t=0$ and $t=t_{0} >0$ are not independent of each other.
 In any case, our technique does not allow us to treat this situation.
 
For the function \eqref{eq:sing1},
let us discuss the operator $\bGank({\bf a}_{m})$ in the representation $\ell^2 (\bbZ_{+})$, that is, the operator
$$
\mathcal{F}\mathcal{U}^* {\bf H} ({\boldsymbol{\tau}}_{m})\mathcal{U}\mathcal{F}^*= \Gamma (a_{m}).
$$
Here $a_{m} (j)$ are the Fourier coefficients of the function $\tau_{m}(\mu)$ linked to ${\boldsymbol{\tau}}_{m}(x)$ by formula   \eqref{eq:HG}. 
Making the change of   variables   \eqref{b8a} in  \eqref{eq:sing6}, we see that $\tau_{m}(\mu)$ is an oscillating function as $\mu\to 1$. Therefore the asymptotics of its Fourier coefficients $a_m(j)$ is determined by the stationary phase method which yields:
  $$
  a_m(j)\sim m! \pi^{-1/2} 2^{- (2m+1)/4} j^{-(2m+5)/4} \cos \big(2\sqrt{2 j} -\pi (2m+1)/4\big).
  $$
  Note that   these sequences decay faster as $j\to\infty$ than the matrix elements
  \eqref{a13}  (for any $\alpha$). Nevertheless due to the    oscillating factor their contribution to the asymptotics of singular values of the Hankel operator $\Gamma(a_{m}+ h)$ is of the same order.

\section*{Acknowledgements}
The authors are grateful to the Departments of Mathematics of King's College London and of the University of Rennes 1 (France)  for the financial support. The second author (D.Y.) acknowledges also the  support and hospitality of the Isaak Newton Institute for Mathematical Sciences (Cambridge  University, UK) where a part of this work has been done during the program Periodic and Ergodic Spectral Problems.


\begin{thebibliography}{17}


\bibitem{BSbook}
{\sc M.~Sh.~Birman, M.~Z.~Solomyak, }
\emph{Spectral theory of selfadjoint operators in Hilbert space.} 
D. Reidel, Dordrecht, 1987. 



\bibitem{BS2}
{\sc M.~Sh.~Birman, M.~Z.~Solomyak, } 
\emph{Compact operators with power asymptotic behavior of the singular numbers.} 
J. Sov. Math. \textbf{27} (1984), 2442--2447. 

 

\bibitem{Part}
{\sc K.~Glover, J.~Lam, J.~R.~Partington,}
\emph{Rational approximation of a class of infinite-dimensional 
systems I: singular values of Hankel operators,}
Math. Control Signals Systems (1990) \textbf{3}, 325--344.


\bibitem{Howland1}
{\sc J.~S.~Howland, }
\emph{Spectral theory   of self-adjoint Hankel matrices}, 
Michigan Math. J.  {\bf 33} (1986), 145--153.



\bibitem{Ismagilov}   
{\sc R.~S.~Ismagilov,} 
{\em   On the spectrum of Toeplitz matrices}, 
Sov. Math. Dokl.  {\bf 4} (1963), 462--465.


\bibitem{NK} 
{\sc N.~K.~Nikolski,} 
{\em Operators, functions, and systems: an easy reading}, vol. I: Hardy, Hankel, and Toeplitz, Math. Surveys and Monographs vol.~92,  Amer. Math. Soc.,   Providence,
Rhode Island, 2002.


\bibitem{Peller}
{\sc V.~Peller,}
\emph{Hankel operators and their applications,}
Springer, 2003.

\bibitem{Power}
{\sc S.~R.~Power,}
\emph{Hankel operators with discontinuous symbols}, 
Proc. Amer. Math. Soc.  \textbf{65} 1977, 77--79. 



 \bibitem{PY1}
{\sc A.~Pushnitski,  D.~Yafaev,}
\emph{Spectral theory of piecewise continuous functions of self-adjoint operators},   
Proc. London Math. Soc.  \textbf{108}  (2014), 1079--1115.


\bibitem{I}
{\sc A.~Pushnitski, D.~Yafaev,}
\emph{Sharp estimates for singular values of Hankel operators},
to appear in Integral Equations and Operator Theory, 
  doi: 10.1007/s00020-015-2239-0.


\bibitem{II}
{\sc A.~Pushnitski, D.~Yafaev,}
\emph{Asymptotic behavior of eigenvalues of Hankel operators},
to appear in Int. Math. Res. Notices, doi: 10.1093/imrn/rnv048.
 
\bibitem{Rat}
{\sc A.~Pushnitski, D.~Yafaev,}
\emph{Best rational approximation of functions with logarithmic singularities},
in preparation.


\bibitem{Yafaev2}
{\sc D.~R.~Yafaev,} 
\emph{Criteria for Hankel operators to be sign-definite,}
Analysis \& PDE \textbf{8} (2015), no. 1, 183--221. 



\end{thebibliography}
\end{document}